\documentclass{amsart}

\usepackage{graphicx,epsfig}

\usepackage{amsmath,amssymb,latexsym, amsfonts, amscd, amsthm, xy}

\usepackage{url}

\usepackage{hyperref}

\input{xy}
\xyoption{all}

\usepackage[usenames]{color}

\makeindex 

=630pt \headheight = 12pt \headsep = 20pt

\theoremstyle{plain}
\newtheorem{theorem}{Theorem}[section]

\newtheorem{proposition}[theorem]{Proposition}
\newtheorem{corollary}[theorem]{Corollary}

\newtheorem{lemma}[theorem]{Lemma}

\newtheorem{question}[theorem]{Question}

\theoremstyle{definition}
\newtheorem{definition}[theorem]{Definition}
\newtheorem{remark}[theorem]{Remark}
\newtheorem{example}[theorem]{Example}

\def\e{\mathbf{e}}

\def\bF{\mathbb{F}}

\def\bZ{\mathbb{Z}}
\def\bR{\mathbb{R}}

\def\D{\mathbf{D}}

\def\cA{\mathcal{A}}

\def\cD{\mathcal{D}}
\def\cE{\mathcal{E}}
\def\cF{\mathcal{F}}

\def\cG{\mathcal{G}}

\def\cL{\mathcal{L}}

\def\cN{\mathcal{N}}
\def\cO{\mathcal{O}}

\def\cQ{\mathcal{Q}}
\def\cS{\mathcal{S}}

\def\cX{\mathcal{X}}

\def\cU{\mathcal{U}}

\def\cT{\mathcal{T}}

\def\fR{\mathfrak{R}}

\def\fp{\mathfrak{p}}

\def\eval{\mathbf{eval}}

\begin{document}

\title{The classical umbral calculus, and the flow of a Drinfeld module}

\author{Nguyen Ngoc Dong Quan}

\date{May 25, 2015}

\address{Department of Mathematics \\
         The University of Texas at Austin \\
         Austin, TX 78712 \\
         USA}

\email{\href{mailto:dongquan.ngoc.nguyen@gmail.com}{\tt dongquan.ngoc.nguyen@gmail.com}}

\maketitle

\tableofcontents

\begin{abstract}

David Goss developed a very general Fourier transform in additive harmonic analysis in the function field setting. In order to introduce the Fourier transform for continuous characteristic $p$ valued functions on $\bZ_p$, Goss introduced and studied an analogue of flows in finite characteristic. In this paper, we use another approach to study flows in finite characteristic. We recast the notion of a flow in the language of the classical umbral calculus, which allows to generalize the formula for flows first proved by Goss to a more general setting. We study duality between flows using the classical umbral calculus, and show that the duality notion introduced by Goss seems a natural one. We also formulate a question of Goss about the exact relationship between two flows of a Drinfeld module in the language of the classical umbral calculus, and give a partial answer to it.

\end{abstract}

\section{Introduction}
\label{Section-Introduction}

The notion of a flow appears in many areas of mathematics and physics; for example, it is one of the fundamental notions in studying ordinary differential equations. Classically a flow on a set $\cX$ is a group action of the additive group of the set of real numbers $\bR$ on the set $\cX$. One of the classical theorems in analysis, the classical Taylor theorem, can be expressed in terms of a flow. More precisely, let $d/dt$ be the formal derivative acting on the polynomial ring $\bR[t]$. For each $x \in \bR$, let $e^{x(d/dt)} : \bR[t] \rightarrow \bR[t]$ be the linear operator defined by
\begin{align*}
e^{x(d/dt)} = \sum_{n = 0}^{\infty}x^n\dfrac{(d/dt)^n}{n!}.
\end{align*}
Then the map $\cF_{d/dt}$ defined by
\begin{align*}
\cF_{d/dt} : \bR[t] \times \bR &\rightarrow \bR[t] \\
(P(t), x) &\mapsto e^{x(d/dt)}P(t)
\end{align*}
is a flow on the polynomial ring $\bR[t]$. The classical Taylor theorem can be recast in the form
\begin{align*}
\cF_{d/dt}(P(t), x) = e^{x(d/dt)}P(t) = P(t + x)
\end{align*}
for any polynomial $P(t) \in \bR[t]$ and any real number $x \in \bR$.

David Goss \cite{Goss-JAlgebra-146-1992} developed a very general Fourier transform in additive harmonic analysis in the function field setting. In order to introduce the Fourier transform for continuous characteristic $p$ valued functions on $\bZ_p$, Goss introduced and studied an analogue of flows in finite characteristic. Let $\bF_q$ denote the finite field with $q$ elements, and let $F = \bF_q((1/t))$ be the completion of the function field $K = \bF_q(t)$. Let $|\cdot|_F$ denote the norm of $F$. Let $F_b[[T]]$ be the algebra of bounded power series (see Definition \ref{Definition-The-definition-of-a-bounded-power-series}). Let $\cF := \{\cF_k(x)\}_{k \ge 0}$ be a sequence of $F$-valued functions defined over $F$ that arises from additive functions (see Remark \ref{Remark-The-Goss-theorem-is-a-special-case} below). For any $x \in F$ with $\lim_{k \rightarrow \infty}\cF_k(x) = 0$, Goss \cite{Goss-JAlgebra-146-1992} defined a \textit{flow} $\cD_{\cF}(x)$ on $F_b[[T]]$ by
\begin{align}
\label{Equation-The-definition-of-flows-in-the-introduction}
\cD_{\cF}(x)P(T) = \sum_{k = 0}^{\infty}\cF_k(x)\dfrac{(d/dT)^k(P(T))}{k!}
\end{align}
for any bounded power series $P(T)$. For example, in the simple case in which $\cF_k(x) = x^k$ for all $k \ge 0$, and $\phi$ is a Drinfeld module with exponential $\e_{\phi}(x)$, Goss obtained the \textit{naive flow} $e^{\e_{\phi}(x)(d/dT)}$ for all $x \in F$ with $|x|_F < 1$, and then derived the flow equation for the Drinfeld module $\phi$ of the form
\begin{align*}
e^{\e_{\phi}(x)(d/dT)}P(T) = P(T + \e_{\phi}(x))
\end{align*}
for any bounded power series $P(T)$, which is a function field analogue of the classical Taylor theorem. There is another flow for a Drinfeld module $\phi$ which Goss called the ``twisted flow'' that is more involved and arises out of the theory of $\fp$-adic measures, where $\fp$ is a prime in $\bF_q[t]$. In order to derive the flow equation for the ``twisted flow", Goss used the differential Fourier transform on measures.

In this paper, we use another approach to study flows in finite characteristic. We recast the notion of a flow in the language of the classical umbral calculus, and then interpret the ``formal substitutions'' appearing in the formula for flows in Goss \cite{Goss-JAlgebra-146-1992} as the images of certain power series having an umbra as one of their variables under the evaluation map of an umbral calculus, which makes the formula for flows \textit{less formal} and more conceptual. Since we only use the classical umbral calculus and do not need the differential Fourier transform on measures as in the proof of \cite[Theorem 1]{Goss-JAlgebra-146-1992}, we can generalize the formula for flows to a more general setting. More precisely, the formula for flows works in a field $F$ of any characteristic provided that $F$ is complete under a non-Archimedean norm. Furthermore the sequences $\{\cF_k(x)\}_{k \ge 0}$ of functions defining flows are not necessary to arise out of additive functions as in Goss \cite{Goss-JAlgebra-146-1992}.

In order to avoid working with collections of functions, we introduce in this paper the notion of an \textit{umbral map} that can view each value of a collection of functions as the evaluation of powers of an indeterminate under a certain linear functional. More explicitly, take any complete field $F$ under a non-Archimedean norm, and let $\cU$ be a set whose elements are called \textit{umbrae}. An umbra in $\cU$ operates in the same way as an indeterminate, and one umbra, as we will see throughout the paper, represents the value of a collection of functions at some element in $F$. To make a transition between umbrae and $F$-valued functions defined over $F$, we assume that the set $\cU$ is equipped with a linear functional $\eval$ defined over $\cU$ and taking values in $F$ satisfying certain conditions.

A map $\cF$ defined over $F$ and taking values in $\cU$ is called an \textit{umbral map.} We can associate to any collection of functions $\{\cF_n(x)\}_{n \ge 0}$ an umbral map $\cF$ by taking an umbra $\cF(x)$ for each $x \in F$ and requiring that $\eval(\cF(x)^n) = \cF_n(x)$ for any integer $n \ge 0$. Hence an umbral map encodes information of its associated collection of functions and vice versa. The corresponding \textit{flow map} of an umbral map $\cF$ is defined by the same equation $(\ref{Equation-The-definition-of-flows-in-the-introduction})$ with $\{\cF_n(x)\}_{n \ge 0}$ being the associated functions of $\cF$. Note that in \cite{Goss-JAlgebra-146-1992}, Goss used the terminology ``flow'' for what we call ``flow map'' in this paper.

In \cite{Goss-JAlgebra-146-1992}, Goss introduced the notion of the dual of a flow, and proved that the image of a flow under the Fourier transform is the dual of the flow. In this paper, we adapt \cite[Theorem 6]{Goss-JAlgebra-146-1992} as the notion of duality between flow maps; more precisely, a flow map $\cD_{\cF}$ is said to be a dual of a flow map $\cD_{\widehat{\cF}}$ if there exists an \textit{additive isomorphism} $\phi : F_b[[T]] \rightarrow F_b[[T]]$ (see Definition \ref{Definition-Additive-linear-isomorphisms}) such that
\begin{align*}
\cD_{\widehat{\cF}}(x) = \phi \circ \cD_{\cF}(x) \circ \phi^{-1}
\end{align*}
for any $x \in F$ with both $\cD_{\cF}(x)$ and $\cD_{\widehat{\cF}}(x)$ being well-defined. One can see immediately from this notion that the dual relation is reflexive and symmetric, which is not clear from the duality notion introduced by Goss. By abuse of terminology, umbral maps whose corresponding flow maps are dual are also said to be \textit{dual}.

Our main result (see Theorem \ref{Theorem-A-flow-under-the-Fourier-transform}) concerning duality between flow maps is that the converse of \cite[Theorem 6]{Goss-JAlgebra-146-1992} is true, and it thus completely describes, in the language of the classical umbral calculus used here, all umbral maps whose corresponding flow maps are dual to each other. The converse of \cite[Theorem 6]{Goss-JAlgebra-146-1992} proved in Section \ref{Section-The-dual-of-a-flow-map} shows that the construction of dual flow maps first introduced by Goss seems a natural and correct one.

By duality, one expects that dual umbral maps should share similar properties, and in this paper we show one of the properties shared by dual umbral maps that seems important, that is, an umbral map \textit{satisfies the binomial theorem} (for a precise definition, see Definition \ref{Definition-Umbrae-satisfy-the-binomial-theorem}) if and only if one of its dual umbral maps satisfies the binomial theorem.

For dual umbral maps $\cF, \widehat{\cF}$, the sequences $\{\cF_n(x)\}_{n \ge 0}$, $\{\widehat{\cF}_n(x)\}_{n \ge 0}$ of functions are associated to $\cF$ and $\widehat{\cF}$, respectively, where $\cF_n(x) = \eval(\cF(x)^n)$ and $\widehat{\cF}_n(x) = \eval(\widehat{\cF}(x)^n)$ for all $x$ and any $n \ge 0$. We introduce another umbral map, denoted by $\cG$ that satisfies the equations $\eval(\cG(x)^n) = \widehat{\cF}_1(x)^n$ for all $x$ and all $n \ge 0$. In \cite{Goss-JAlgebra-146-1992}, Goss asked what the exact relationship between the flow maps $\cD_{\cG}$ and $\cD_{\widehat{\cF}}$ is. In this paper, we partially answer the question of Goss, and prove a necessary and sufficient criterion on the values of $\cF$ for which the flow maps $\cD_{\cG}$ and $\cD_{\widehat{\cF}}$ agree on a certain subset of the underlying field.

The paper is organized as follows. In Section \ref{Section-The-classical-umbral-calculus}, we recall the classical umbral calculus following Rota \cite{Rota-AMM-71-1964}, Roman and Rota \cite{Roman-Rota}, and Rota and Taylor \cite{Rota-Taylor-1993} \cite{Rota-Taylor-SIAM-25-1994}, and point out some slight modifications that we make in our description of the classical umbral calculus. In Section \ref{Section-The-flow-of-an-umbra}, we recast the notion of a flow in finite characteristic as introduced in \cite{Goss-JAlgebra-146-1992} in the language of the classical umbral calculus. Since our approach only uses the formalism of the classical umbral calculus, we can generalize the flow equations for Drinfeld modules in \cite{Goss-JAlgebra-146-1992} to a more general setting in which the domain of umbrae can be a field of any characteristic that is complete with respect to a non-Archimedean norm, and umbrae do not necessarily arise out of additive functions. In Section \ref{Section-The-dual-of-a-flow-map}, we recast the notion of duality between flow maps using the formalism of the classical umbral calculus. The main theorem in this section (see Theorem \ref{Theorem-A-flow-under-the-Fourier-transform}) is the converse of \cite[Theorem 6]{Goss-JAlgebra-146-1992} that completely describes the relationship between dual umbral maps. In Section \ref{Section-Examples-and-a-question-of-Goss}, we formulate the question of Goss in the language of the classical umbral calculus, and give a partial answer to it. We also discuss some examples of flow maps in this section. In the last section, we compare our umbral calculus with other umbral calculi in literature, for example, in Rota \cite{Rota-AMM-71-1964}, Roman and Rota \cite{Roman-Rota}, Rota and Taylor \cite{Rota-Taylor-1993} \cite{Rota-Taylor-SIAM-25-1994}, and Ueno \cite{Ueno-1988}. 

\section{The classical umbral calculus}
\label{Section-The-classical-umbral-calculus}

We give a description of the classical umbral calculus following Rota \cite{Rota-AMM-71-1964}, Roman and Rota \cite{Roman-Rota}, and Rota and Taylor \cite{Rota-Taylor-1993} \cite{Rota-Taylor-SIAM-25-1994} with slight modifications. In the traditional approach of the classical umbral calculus as presented in \cite{Rota-AMM-71-1964}, \cite{Rota-Taylor-1993}, \cite{Rota-Taylor-1993}, and \cite{Rota-Taylor-SIAM-25-1994}, the domain of umbrae is an integral domain of characteristic zero. In contrast to the traditional treatment, we work with a field of any characteristic as the domain of umbrae, and in Sections \ref{Section-The-dual-of-a-flow-map} and \ref{Section-Examples-and-a-question-of-Goss} we assume that the domain of umbrae is a field of characteristic $p > 0$.

We will mainly work with power series whose coefficients are power series having umbrae as their variables. Hence in order to extend the evaluation map of an umbral calculus to evaluate such objects, a necessary condition is that the domain of umbrae is a complete field with respect to a non-Archimedean norm. We now begin to describe the classical umbral calculus that will be used throughout the paper.

Let $F$ be a field such that $F$ is complete with respect to a non-Archimedean norm $|\cdot|_F$. The characteristic of $F$ is arbitrary unless otherwise stated. Let $F[[T]]$ denote the formal power series in the variable $T$ over $F$. An \textit{umbral calculus} on $F[[T]]$ is a pair $(\cU, \eval)$ satisfying the following conditions.
\begin{itemize}

\item [(1)] $\cU$ is a set whose elements are called \textit{umbrae};

\item [(2)] $\eval : F[[T]][\cU] \rightarrow F[[T]]$ is a linear functional defined on the polynomial ring $F[[T]][\cU]$ with values in $F[[T]]$ satisfying:
\begin{itemize}

\item [(i)] $\eval(1) = 1$, where $1$ is the identity element of $F$;

\item [(ii)] $\eval(\alpha^n) \in F$ for any umbra $\alpha \in \cU$ and any nonnegative integer $n$; and

\item [(iii)] $\eval(\alpha_1^{n_1}\alpha_2^{n_2}\cdots\alpha_m^{n_m}) = \eval(\alpha_1^{n_1})\eval(\alpha_2^{n_2})\cdots \eval(\alpha_m^{n_m})$, where $\{\alpha_1, \ldots, \alpha_m\}$ is an arbitrary collection of \textit{distinct} umbrae and $n_1, n_2, \ldots, n_m$ are arbitrary nonnegative integers.

\end{itemize}

\end{itemize}

Note that Rota and Taylor \cite{Rota-Taylor-SIAM-25-1994} also included the \textit{augmentation} in the definition of an umbral calculus, namely, an umbra $\epsilon$ such that $\eval(\epsilon^n) = \delta_{n, 0}$, where $\delta_{n, 0}$ is the Kronecker delta. We however do not include the augmentation in the above definition. One reason is that we will not need it here, or elsewhere in the paper. Another reason is that because of condition $(2)(iii)$ above, i.e., the independence between the values of powers of distinct umbrae, we always can \textit{augment} the set $\cU$ to include the augmentation if necessary. 

For the rest of this paper, we will always denote by $(\cU, \eval)$ an umbral calculus on $F[[T]]$, and assume that the set $\cU$ is sufficiently large to contain all umbrae in this paper. The action of the evaluation map $\eval$ on any umbra in this paper will be clear from the context.

Let $\cS$ be a finite subset of $\cU$ that consists of distinct umbrae $\gamma_1, \gamma_2, \ldots, \gamma_s$. We want to extend the linear functional $\eval$ to a certain subset of the algebra of power series $F[[T]][[\cS]]$. It turns out that $\eval$ can extend to a subalgebra of $F[[T]][[\cS]]$ whose elements are called \textit{admissible power series}.
\begin{definition}
\label{Definition-Admissible-power-series}

Let $P(\gamma_1, \gamma_2, \ldots, \gamma_s) = \sum_{n_1, n_2, \ldots, n_s \ge 0}P_{n_1, \ldots, n_s}\gamma_1^{n_1}\gamma_2^{n_2}\cdots \gamma_s^{n_s}$ be a power series in $F[[T]][[\cS]]$, where the $P_{n_1, \ldots, n_s}$ are elements in $F[[T]]$. Write $P(\gamma_1, \gamma_2, \ldots, \gamma_s)$ in the form
\begin{align*}
P(\gamma_1, \gamma_2, \ldots, \gamma_s) = \sum_{n = 0}^{\infty}U_n(\gamma_1, \ldots, \gamma_s) T^n,
\end{align*}
where the $U_n(\gamma_1, \ldots, \gamma_s)$ are elements in $F[[\cS]]$. We say that $P(\gamma_1, \gamma_2, \ldots, \gamma_s)$ is an \textit{admissible power series} if $\eval\left(U_n(\gamma_1, \cdots, \gamma_s)\right)$ is well-defined as an element of $F$ for all $n \ge 0$.

\end{definition}

\begin{remark}
\label{Remark-Clarification-of-the-notion-of-admissible-power-series}

We maintain the same notation as in Definition \ref{Definition-Admissible-power-series}. Take any integer $n \ge 0$, and write
\begin{align*}
U_n(\gamma_1, \ldots, \gamma_s) = \sum_{m_1, m_2, \ldots, m_s \ge 0}a_{m_1, m_2, \ldots, m_s}\gamma_1^{m_1}\gamma_2^{m_2}\cdots \gamma_s^{m_s},
\end{align*}
where the $a_{m_1, m_2, \ldots, m_s}$ are elements in $F$. Since $F$ is a complete field under a non-Archimedean norm, we see that $\eval\left(U_n(\gamma_1, \cdots, \gamma_s)\right)$ is well-defined as an element of $F$ if either of the following conditions is satisfied.
\begin{itemize}

\item [(i)] $U_n(\gamma_1, \ldots, \gamma_s)$ is a polynomial in the variables $\gamma_1, \ldots, \gamma_s$ over $F$, that is, all but finitely many coefficients $a_{m_1, m_2, \ldots, m_s}$ are zero. (Note that Gessel \cite{Gessel} used this condition as the notion of admissible power series.)

\item [(ii)] if $U_n(\gamma_1, \ldots, \gamma_s)$ is not a polynomial, that is, there are infinitely many nonzero coefficients $a_{m_1, m_2, \ldots, m_s}$, then
    \begin{align*}
    a_{m_1, m_2, \ldots, m_s}\eval(\gamma_1^{m_1})\eval(\gamma_2^{m_2})\cdots \eval(\gamma_s^{m_s}) \rightarrow 0
    \end{align*}
    when $m := m_1 + m_2 + \cdots + m_s \rightarrow \infty$.

\end{itemize}

\end{remark}

We denote by $F[[T]][[\cS]]_{\cA}$ the subset of $F[[T]][[\cS]]$ that consists of all admissible power series. The following result is immediate from Definition \ref{Definition-Admissible-power-series} and Remark \ref{Remark-Clarification-of-the-notion-of-admissible-power-series}.

\begin{proposition}
\label{Proposition-Ring-structure-of-F[[T]][[S]]-A}

$F[[T]][[\cS]]_{\cA}$ is a subalgebra of $F[[T]][[\cS]]$.

\end{proposition}

It is obvious that the polynomial ring $F[[T]][\cS]$ is a subalgebra of $F[[T]][[\cS]]_{\cA}$. If $P(\gamma_1, \gamma_2, \ldots, \gamma_s)$ is an admissible power series in $F[[T]][[\cS]]_{\cA}$, we define
\begin{align}
\label{Definition-The-evaluation-of-admissible-power-series}
\eval(P) = \sum_{n = 0}^{\infty}\eval\left(U_n(\gamma_1, \ldots, \gamma_s)\right) T^n,
\end{align}
where the $U_n(\gamma_1, \ldots, \gamma_s)$ are the unique elements in $F[[\cS]]$ such that
\begin{align*}
P(\gamma_1, \gamma_2, \ldots, \gamma_s) = \sum_{n = 0}^{\infty}U_n(\gamma_1, \ldots, \gamma_s) T^n.
\end{align*}
Since $P$ is an admissible power series, the right-hand side of $(\ref{Definition-The-evaluation-of-admissible-power-series})$ is well-defined as an element of $F[[T]]$. Hence $\eval$ can be extended to a linear functional, also denoted by $\eval$, defined on the algebra of admissible power series $F[[T]][[\cS]]_{\cA}$ and taking values in $F[[T]]$.

Let $\{u_k\}_{k \ge 0}$ be a sequence of elements in $F$. An umbra $\alpha$ is said to \textit{umbrally represent} $\{u_k\}_{k \ge 0}$ if $\eval(\alpha^k) = u_k$ for all $k \ge 0$. Since $\eval(\alpha^0) = \eval(1) = 1$, it is necessary that $u_0 = 1$ if the sequence $\{u_k\}_{k \ge 0}$ is umbrally represented by $\alpha$.

For an element $P \in F[[T]][[\cS]]$, we can write $P$ as an (infinite) sum of distinct monomials with nonzero coefficients in $F[[T]]$. The \textit{support} of $P$ consists of all umbrae that occur in some monomial with positive power in the sum representation of $P$.

Let $P \in F[[T]][[\cS]]$. If $\{\gamma_{i_1}, \gamma_{i_2}, \ldots, \gamma_{i_k}\} \subset \cS$ is the support of $P$, we can write $P$ as $P(\gamma_{i_1}, \ldots, \gamma_{i_k})$ to signify that $P$ only depends on the umbrae $\gamma_{i_1}, \ldots, \gamma_{i_k}$.

\begin{definition}
\label{Definition-The-definition-of-a-bounded-power-series}

A formal power series $P(T) = \sum_{m = 0}^{\infty}a_m T^m \in F[[T]]$ is said to be \textit{bounded} if there exists a positive constant $c > 0$ such that $|a_m|_F < c$ for all $m \ge 0$.

\end{definition}

It is not difficult to see that the set of all bounded formal power series in $F[[T]]$ forms a subalgebra of $F[[T]]$. We denote by $F_b[[T]]$ the algebra of all bounded formal power series.

Throughout the paper, we denote by $D = d/dT$ the derivation acting on $F[[T]]$, that is, $DT^k = kT^{k - 1}$ for all $k \ge 1$, and extend $D$ by linearity. For all $n \ge 1$, we define $D^n = D(D^{n - 1})$ inductively, where $D^1 = D$ and $D^0$ is the identity map.

Recall that the \textit{$k$-th Hasse derivative} $\D^{(k)}$ is defined by $\dfrac{D^k}{k!}$, that is, for any $n \ge 1$ and any $k \ge 0$,
\begin{align*}
\D^{(k)}T^n =
\begin{cases}
\binom{n}{k}T^{n - k} \; \; &\text{if $n \ge k$,} \\
0 \; \; &\text{otherwise.}
\end{cases}
\end{align*}

\begin{lemma}
\label{Lemma-Derivatives-of-a-bounded-power-series-is-bounded}

Let $P(T)$ be a bounded power series in $F_b[[T]]$. Then for any $k \ge 0$, the formal power series $\D^{(k)}P(T)$ is bounded.

\end{lemma}

\begin{proof}

Take any integer $k \ge 0$, and write $P(T) = \sum_{m = 0}^{\infty}a_mT^m$. We see that
\begin{align*}
\D^{(k)}P(T) = \sum_{m \ge k}a_m\binom{m}{k}T^{m - k}.
\end{align*}

By assumption, there exists a positive constant $c > 0$ such that $|a_m|_F < c$ for all $m \ge 0$. Since $|\cdot|_F$ is non-Archimedean, we deduce that
\begin{align*}
\left|\binom{m}{k}a_{m}\right|_F = |\underbrace{a_{m} + a_{m} + \dotsb + a_{m}}_{\binom{m}{k} \; \text{copies of $a_m$}}|_F \le \max(\underbrace{|a_{m}|_F, \dotsc, |a_{m}|_F}_{\binom{m}{k} \; \text{copies of $a_m$}}) = |a_{m}|_F < c
\end{align*}
for all $m \ge k$. Thus $\D^{(k)}P(T)$ is a bounded power series for all $k \ge 0$.

\end{proof}

\begin{proposition}
\label{Proposition-A-condition-for-a-bounded-power-series-to-be-admissible}

Let $P(T)$ be a bounded power series in $F_b[[T]]$. Let $\alpha$ be an umbra, and let $\{u_k\}_{k \ge 0} \subset F$ be the sequence umbrally represented by $\alpha$. Assume that $u_k \rightarrow 0$ when $k \rightarrow \infty$. Then the umbral power series $P(T + \alpha) \in F[[T]][[\alpha]]$ is an admissible power series.

\end{proposition}

\begin{proof}

We can write $P(T + \alpha)$ in the form
\begin{align*}
P(T + \alpha) = \sum_{k = 0}^{\infty}\D^{(k)}(P(\alpha))T^k.
\end{align*}
Write $P(T) = \sum_{m = 0}^{\infty}a_mT^m$. We see that
\begin{align*}
\D^{(k)}P(T) = \sum_{m \ge k}a_m\binom{m}{k}T^{m - k}
\end{align*}
for all $k \ge 0$, and hence $\D^{(k)}P(\alpha) = \sum_{m \ge k}a_m\binom{m}{k}\alpha^{m - k}$.

Take any integer $k \ge 0$. We know from Lemma \ref{Lemma-Derivatives-of-a-bounded-power-series-is-bounded} that $\D^{(k)}P(T)$ is bounded, and hence there exists a constant $c > 0$ such that
\begin{align*}
\left|a_m\binom{m}{k}\right|_F < c
\end{align*}
for all $m \ge k$. Since $u_{m - k} \rightarrow 0$ when $m \rightarrow \infty$, it follows that $a_m\binom{m}{k}u_{m - k} \rightarrow 0$ when $m \rightarrow \infty$. Since $F$ is a complete field with respect to the non-Archimedean norm $|\cdot|_F$, it follows that the sum defined by
\begin{align*}
\sum_{m \ge k}a_m\binom{m}{k}\eval(\alpha^{m - k}) = \sum_{m \ge k}a_m\binom{m}{k}u_{m - k}
\end{align*}
converges to an element in $F$. It thus follows from Definition \ref{Definition-Admissible-power-series} that $P(T + \alpha)$ is an admissible power series in $F[[T]][[\alpha]]_{\cA}$.

\end{proof}

\subsection{The exponential operator}
\label{Subsection-The-exponential-operator}

For each umbra $\alpha$, we define a linear operator $\cE^{\alpha \D} : F[[T]] \rightarrow F[[T]][[\alpha]]$ by the equation
\begin{align*}
\cE^{\alpha \D} = \sum_{k = 0}^{\infty}\alpha^k \dfrac{(d/dT)^k}{k!} = \sum_{k = 0}^{\infty}\alpha^k \D^{(k)}.
\end{align*}
More explicitly, we have
\begin{align*}
\cE^{\alpha \D}P(T) = \sum_{k = 0}^{\infty}\alpha^k \D^{(k)}P(T)
\end{align*}
for any formal power series $P(T) \in F[[T]]$. We call $\cE^{\alpha \D}$ the \textit{exponential operator with multiplicity $\alpha$}.

The following result is an analogue of the Taylor theorem for formal power series.
\begin{theorem}
\label{Theorem-The-theorem-of-Taylor-for-formal-power-series}

Let $P(T) \in F[[T]]$, and let $\alpha$ be an umbra. Then
\begin{align*}
\cE^{\alpha \D}P(T) = P(T + \alpha).
\end{align*}

\end{theorem}

\begin{proof}

By definition, we know that the coefficient of $\alpha^k$ in $P(T + \alpha)$ is $\D^{(k)}P(T)$, and hence Theorem \ref{Theorem-The-theorem-of-Taylor-for-formal-power-series} follows immediately.

\end{proof}

\begin{corollary}
\label{Corollary-The-image-of-the-ring-of-bounded-power-series-under-the-exponential-operator}

Let $\alpha$ be an umbra, and let $\{u_k\}_{k \ge 0}$ be the sequence umbrally represented by $\alpha$. Assume that $u_k \rightarrow 0$ when $k \rightarrow \infty$. Then the image of $F_b[[T]]$ under the exponential operator $\cE^{\alpha \D}$ is a subset of $F[[T]][[\alpha]]_{\cA}$: that is, $\cE^{\alpha \D}P(T)$ is an admissible power series for any bounded power series $P(T) \in F_b[[T]]$.

\end{corollary}

\begin{proof}

Take any bounded power series $P(T) \in F_b[[T]]$. By Theorem \ref{Theorem-The-theorem-of-Taylor-for-formal-power-series}, we know that $\cE^{\alpha\D}P(T) = P(T + \alpha)$. Hence it follows from Proposition \ref{Proposition-A-condition-for-a-bounded-power-series-to-be-admissible} that $\cE^{\alpha\D}P(T) = P(T + \alpha)$ is an admissible power series.

\end{proof}

\section{The flow of an umbra}
\label{Section-The-flow-of-an-umbra}

Let $\cF : F \rightarrow \cU$ be a map defined over $F$ and taking values in $\cU$; in other words, $\cF(x)$ is an umbra for all $x \in F$. We call $\cF$ the \textit{umbral map}. For each $x \in F$, we will always denote by $\{\cF_k(x)\}_{k \ge 0}$ the sequence umbrally represented by the umbra $\cF(x)$.

Define
\begin{align*}
\cA_{\cF} := \{x \in F \; | \; \text{$\cF_k(x) \rightarrow 0$ when $k \rightarrow \infty$}\}.
\end{align*}
We call $\cA_{\cF}$ the set of \textit{$\cF$-admissible elements}.

Fix an element $x \in \cA_{\cF}$. By Corollary \ref{Corollary-The-image-of-the-ring-of-bounded-power-series-under-the-exponential-operator}, we see that $\cE^{\cF(x)\cD}P(T)$ is an admissible power series in $F[[T]][[\cF(x)]]_{\cA}$ for any bounded power series $P(T) \in F_b[[T]]$. Thus $\eval(\cE^{\cF(x)\D}P(T))$ is well-defined as an element in $F[[T]]$ for any bounded power series $P(T) \in F_b[[T]]$.

By definition, we see that
\begin{align*}
\eval\left(\cE^{\cF(x)\D}P(T)\right) = \eval\left(\sum_{k = 0}^{\infty}\cF(x)^k \D^{(k)}P(T)\right) = \sum_{k = 0}^{\infty}\eval\left(\cF(x)^k\right)\D^{(k)}P(T) = \sum_{k = 0}^{\infty}\cF_k(x)\D^{(k)}P(T)
\end{align*}
for any bounded power series $P(T)$ in $F_b[[T]]$. Hence $\sum_{k = 0}^{\infty}\cF_k(x)\D^{(k)}P(T)$ is well-defined as an element in $F[[T]]$ for any bounded power series $P(T)$. This motivates the following definitions.

\begin{definition}
\label{Definition-The-flow-operator}

For each $x \in \cA_{\cF}$, we define $\cD_{\cF}(x) : F_b[[T]] \rightarrow F[[T]]$ by
\begin{align*}
\cD_{\cF}(x) = \sum_{k = 0}^{\infty}\cF_k(x) \D^{(k)}.
\end{align*}

\end{definition}

\begin{definition}
\label{Definition-The-flow-of-an-umbra}

\begin{itemize}

\item []

\item [(i)] For each $x \in \cA_{\cF}$, the action of $\cD_{\cF}(x)$ on $F_b[[T]]$ is called the \textit{flow of the umbral $\cF(x)$.}

\item [(ii)] The map $\cD_{\cF}$ that sends each element $x \in \cA_{\cF}$ to the linear operator $\cD_{\cF}(x) : F_b[[T]] \rightarrow F[[T]]$ is called the \textit{flow map} of $\cF$.

\end{itemize}

\end{definition}

In Corollary \ref{Corollary-A-flow-is-an-action-on-the-algebra-of-bounded-power-series} below, we prove that the image of a flow map $\cD_{\cF}$ is a subset of the ring of linear operators from $F_b[[T]]$ to itself. In order words, the image of a bounded power series under $\cD_{\cF}(x)$ is bounded for any $x \in \cA_{\cF}$.

We now prove the main result of this section that is a generalization of \cite[Theorem 1]{Goss-JAlgebra-146-1992}.

\begin{theorem}
\label{Theorem-A-generalization-of-the-Goss-theorem}

Let $x$ be an element in $\cA_{\cF}$. Then
\begin{align*}
\cD_{\cF}(x)P(T) = \eval(P(T + \cF(x))) = \sum_{k = 0}^{\infty}a_k\left(\sum_{m = 0}^k \binom{k}{m}\cF_m(x)T^{k - m}\right)
\end{align*}
for any bounded power series $P(T) = \sum_{k = 0}^{\infty} a_kT^k \in F_b[[T]]$.

\end{theorem}

\begin{proof}

By the discussion preceding Definition \ref{Definition-The-flow-operator} and Theorem \ref{Theorem-The-theorem-of-Taylor-for-formal-power-series}, we know that
\begin{align*}
\cD_{\cF}(x)P(T) &= \eval(\cE^{\cF(x)\D}P(T)) \\
&= \eval(P(T + \cF(x))) \\
&= \eval\left(\sum_{k = 0}^{\infty}a_k(T + \cF(x))^k \right) \\
&= \eval\left(\sum_{k = 0}^{\infty}a_k\left(\sum_{m = 0}^k\binom{k}{m}\cF(x)^mT^{k -m}\right) \right) \\
&= \sum_{k = 0}^{\infty}a_k\left(\sum_{m = 0}^k\binom{k}{m}\eval\left(\cF(x)^m\right)T^{k -m}\right) \\
&= \sum_{k = 0}^{\infty}a_k\left(\sum_{m = 0}^k\binom{k}{m}\cF_m(x)T^{k -m}\right),
\end{align*}
which proves our contention.

\end{proof}

\begin{remark}
\label{Remark-The-Goss-theorem-is-a-special-case}

In \cite{Goss-JAlgebra-146-1992}, Goss proved a special case of Theorem \ref{Theorem-A-generalization-of-the-Goss-theorem} in which $F$ is the completion of a function field at infinity with the additional assumption that umbrae arise out of additive functions. More precisely, Goss considered the completion of the function field $K = \bF_q(t)$ at $1/t$, say $F = \bF_q((1/t))$, where $q$ is a power of a prime $p$. \textit{D'apr\`es} Carlitz \cite{Carlitz-Duke-Math-Journal-1940-no.6}, Goss \cite{Goss-JAlgebra-146-1992} introduced a sequence of functions $\{\cF_k(x)\}_{k \ge 0}$ as follows. Take any sequence of \textit{additive} functions $\{\epsilon_k(x)\}_{k \ge 0}$. For each positive integer $k$, write $k$ $q$-adically as $k = \sum_{i = 0}^{h}\alpha_i q^i$ with $0 \le \alpha_i < q$, and define
\begin{align*}
\cF_k(x) := \prod_{i = 0}^{h}\epsilon_i(x)^{\alpha_i}.
\end{align*}
In the language of the classical umbra calculus introduced in Section \ref{Section-The-classical-umbral-calculus}, Goss considered the umbra map $\cF : F \rightarrow \cU$ arising from additive functions that is defined by $\eval(\cF(x)^k) = \cF_k(x)$ for each $x \in F$ and all $k \ge 0$. With this assumption on umbrae and the field $F$ defined as above, Goss derived the flow equations of the form as presented in Theorem \ref{Theorem-A-generalization-of-the-Goss-theorem}. In Theorem \ref{Theorem-A-generalization-of-the-Goss-theorem}, we remove the additivity assumptions on umbrae, and only assume that $F$ is a field of arbitrary characteristic that is complete with respect to a non-Archimedean norm.

\end{remark}

\begin{remark}

Goss proved \cite[Theorem 1]{Goss-JAlgebra-146-1992} using the theory of non-Archimedean measures, and the assumption that the umbra $\cF$ \textit{satisfies the binomial theorem} (see Definition \ref{Definition-Umbrae-satisfy-the-binomial-theorem}) is one of the key ingredients in the proof of \cite[Theorem 1]{Goss-JAlgebra-146-1992}. The proof of Theorem \ref{Theorem-A-generalization-of-the-Goss-theorem} was carried out using only the formalism of the classical umbral calculus.

\end{remark}

\begin{corollary}
\label{Corollary-A-flow-is-an-action-on-the-algebra-of-bounded-power-series}

Let $\cF : F \rightarrow \cU$ be an umbral map, and let $x \in \cA_{\cF}$. Then $\cD_{\cF}(x)P(T)$ is bounded for any bounded power series $P(T) \in F_b[[T]]$. In other words, $\cD_{\cF}(x)$ is defined over $F_b[[T]]$ and takes values in $F_b[[T]]$.

\end{corollary}

\begin{proof}

Let $P(T) = \sum_{k = 0}^{\infty}a_kT^k$ be an arbitrary bounded power series. Since $P(T)$ is bounded, there exists an absolute constant $c_1 > 0$ such that $|a_k|_F < c_1$ for all $k \ge 0$. Let $\{\cF_m(x)\}_{m \ge 0}$ be the sequence umbrally represented by $\cF(x)$. Since $\cF_m(x) \rightarrow 0$ when $m \rightarrow \infty$, we deduce that there exists an absolute constant $c_2 > 0$ such that $|\cF_m(x)|_F < c_2$ for all $m \ge 0$.

For each $h \ge 0$, let $\epsilon_h = \eval\left(\D^{(h)}P(\cF(x))\right)$. We prove that $|\epsilon_h|_F < c_1c_2$ for all $h \ge 0$. Indeed, we know that
\begin{align*}
\D^{(h)}P(T) = \sum_{k \ge h}\binom{k}{h}a_kT^{k - h} = \sum_{k = 0}^{\infty}\binom{k + h}{h}a_{k + h}T^k,
\end{align*}
and thus
\begin{align}
\label{Equation-The-equation-of-epsilon-h-in-terms-of-a-power-series}
\epsilon_h = \eval\left(\D^{(h)}P(\cF(x))\right) = \eval\left(\sum_{k = 0}^{\infty}\binom{k + h}{h}a_{k + h}\cF(x)^k\right) = \sum_{k = 0}^{\infty}\binom{k + h}{h}a_{k + h}\cF_k(x)
\end{align}
for all $h \ge 0$. Note that since $x$ belongs to $\cA_F$ and $|\binom{k + h}{h}a_{k + h}|_F \le |a_{k + h}|_F < c_1$ for any $h, k \ge 0$, it follows that $\binom{k + h}{h}a_{k + h}\cF_k(x) \rightarrow 0$ when $k \rightarrow \infty$. Thus $\epsilon_h$ is well-defined as an element in $F$ for all $h \ge 0$.

Let $h$ be any nonnegative integer such that $\epsilon_h \ne 0$, and hence $|\epsilon_h|_F > 0$. By equation $(\ref{Equation-The-equation-of-epsilon-h-in-terms-of-a-power-series})$, we see that there exists a sufficiently large integer $k_0 > 0$ such that
\begin{align*}
\left|\sum_{s = 0}^{k}\binom{s + h}{h}a_{s + h}\cF_s(x) - \epsilon_h\right|_F < |\epsilon_h|_F
\end{align*}
for all $k > k_0$. Since
\begin{align*}
\sum_{s = 0}^{k}\binom{s + h}{h}a_{s + h}\cF_s(x) = \left(\sum_{s = 0}^{k}\binom{s + h}{h}a_{s + h}\cF_s(x) - \epsilon_h\right) + \epsilon_h,
\end{align*}
we deduce from the above inequality that
\begin{align*}
\left|\sum_{s = 0}^{k}\binom{s + h}{h}a_{s + h}\cF_s(x)\right|_F = |\epsilon_h|_F
\end{align*}
for all $k > k_0$.

Take any integer $k > k_0$. Since
\begin{align*}
\left|\binom{s + h}{h}a_{s + h}\right|_F\left|\cF_s(x)\right|_F < c_1c_2
\end{align*}
for any integer $s \ge 0$, we deduce that
\begin{align*}
|\epsilon_h|_F &= \left|\sum_{s = 0}^{k}\binom{s + h}{h}a_{s + h}\cF_s(x)\right|_F \\
&\le \max_{0 \le s \le k}\left(\left|\binom{s + h}{h}a_{s + h}\cF_s(x)\right|_F\right) \\
&\le \max_{0 \le s \le k}\left(\left|\binom{s + h}{h}a_{s + h}\right|_F\left|\cF_s(x)\right|_F\right) \\
&< \max_{0 \le s \le k}\left(c_1c_2\right) \\
&= c_1c_2.
\end{align*}
Since $h$ is an arbitrary integer such that $\epsilon_h \ne 0$, the above identity implies that $|\epsilon_h|_F < c_1c_2$ for all $h \ge 0$.

By Theorem \ref{Theorem-A-generalization-of-the-Goss-theorem}, we know that
\begin{align*}
\cD_{\cF}(x)P(T) &= \eval(P(T + \cF(x)) \\
&= \eval\left(\sum_{h = 0}^{\infty}\D^{(h)}(P(\cF(x)))T^h\right) \\
&= \sum_{h = 0}^{\infty}\eval\left(\D^{(h)}(P(\cF(x)))\right)T^h \\
&= \sum_{h = 0}^{\infty}\epsilon_hT^h,
\end{align*}
which implies that $\cD_{\cF}(x)P(T)$ is bounded.

\end{proof}

\section{Duality between flow maps}
\label{Section-The-dual-of-a-flow-map}

Throughout this section, we assume further that $F$ is of characteristic $p > 0$. We denote by $\cO_F$ the ring of all elements $x \in F$ with $|x|_F \le 1$. The following definition plays a key role in studying duality between flow maps.

\begin{definition}
\label{Definition-Additive-linear-isomorphisms}

An algebra isomorphism $\phi : F[[T]] \rightarrow F[[T]]$ is \textit{additive} if the power series $H(T)$ defined by $H(T) = \phi(T) \in F[[T]]$ satisfies the following.
\begin{itemize}

\item [(i)] $H(T)$ is bounded;

\item [(ii)] $H(T)$ is additive, that is, it only involves powers $T^{p^{m}}$; and

\item [(iii)] $\D(H(T))$ belongs to $\cO_F^{\times}$, i.e., $|\D(H(T))|_F = 1$, where $\D$ is the Hasse derivative.

\end{itemize}

When $\phi$ is an additive isomorphism, we call $H(T)$ the \textit{generator} of $\phi$. Note that since $\phi$ is an additive isomorphism, the composition inverse of $H(T)$, denoted by $H^{(-1)}(T)$ is the generator of $\phi^{-1}$.

\end{definition}

We recall the following well-known result about Hasse derivatives of powers of a power series whose proof can be found, for example, in \cite{Jeong-JNT-131-6}.

\begin{proposition}
\label{Proposition-Power-rule-for-powers-of-a-power-series}

Let $P(T)$ be a formal power series in $F[[T]]$. Then
\begin{align*}
\D^{(n)}(P(T)^k) = \sum_{h = 1}^{k}\binom{k}{h}P(T)^{k - h}\left(\sum_{\substack{i_1, \ldots, i_h \ge 1 \\ i_1 + \cdots + i_h = n}}\D^{(i_1)}(P(T))\cdots \D^{(i_h)}(P(T)) \right)
\end{align*}
for any $n \ge 1$ and $k \ge 2$.

\end{proposition}

\begin{lemma}
\label{Lemma-The-restriction-of-an-additive-isomorphism-to-the-ring-of-bounded-power-series}

Let $\phi : F[[T]] \rightarrow F[[T]]$ be an additive isomorphism. Then $\phi(P(T))$ is bounded for any bounded power series $P(T) \in F_b[[T]]$. In other words, the restriction of $\phi$ to $F_b[[T]]$ is an automorphism of $F_b[[T]]$.

\end{lemma}

\begin{proof}

Let $P(T) = \sum_{k = 0}^{\infty}a_kT^k \in F_b[[T]]$. We see that
\begin{align}
\label{Equation-The-equation-for-phi(P(T))}
\phi(P(T)) = \phi\left(\sum_{k = 0}^{\infty}a_kT^k\right) = \sum_{k = 0}^{\infty}a_k\phi(T)^k = \sum_{k = 0}^{\infty}a_kH(T)^k = P(H(T)),
\end{align}
where $H(T)$ is the generator of $\phi$.

Define $\gamma = \D(H(T)) \in \cO_F^{\times}$, and take any integer $n \ge 1$ and $k \ge 2$. By Proposition \ref{Proposition-Power-rule-for-powers-of-a-power-series}, we have that
\begin{align}
\label{Equation-The-1st-equation-the-lemma-about-the-restriction-of-an-additive-isomorphism}
\D^{(n)}(H(T)^k) = \sum_{h = 1}^{k}\binom{k}{h}H(T)^{k - h}\left(\sum_{\substack{i_1, \ldots, i_h \ge 1 \\ i_1 + \cdots + i_h = n}}\D^{(i_1)}(H(T))\cdots \D^{(i_h)}(H(T)) \right).
\end{align}
Since $H(0) = 0$, it follows that
\begin{align}
\label{Equation-The-2nd-equation-the-lemma-about-the-restriction-of-an-additive-isomorphism}
\D^{(n)}(H(T)^k)_{|T = 0} = \sum_{\substack{i_1, \ldots, i_k \ge 1 \\ i_1 + \cdots + i_k = n}}\D^{(i_1)}(H(T))_{|T = 0}\cdots \D^{(i_k)}(H(T))_{|T = 0} .
\end{align}

Since $H(T)$ is the generator of $\phi$, we see that $\D(H(T)) = \gamma$ belongs to $\cO_F^{\times}$, and $\D^{(n)}(H(T)) = 0$ for all $n \ge 2$. Hence the terms $\D^{(i_1)}(H(T))_{|T = 0}\cdots \D^{(i_k)}(H(T))_{|T = 0}$ in the sum on the right-hand side of $(\ref{Equation-The-1st-equation-the-lemma-about-the-restriction-of-an-additive-isomorphism})$ is nonzero if and only if $i_1 = i_2 = \ldots = i_k = 1$. This implies that $i_1 + i_2 + \cdots + i_k = k$, and hence for any $n \ge 1$ and $k \ge 2$, we deduce from $(\ref{Equation-The-1st-equation-the-lemma-about-the-restriction-of-an-additive-isomorphism})$ that
\begin{align}
\label{Equation-The-Hasse-derivatives-of-H(T)}
\D^{(n)}(H(T)^k)_{|T = 0} =
\begin{cases}
\gamma^k \; \; &\text{if $k = n$,} \\
0 \; \; &\text{if $k \ne n$.}
\end{cases}
\end{align}

Write $\phi(P(T)) = \sum_{n = 0}^{\infty}\epsilon_n T^n$, where the $\epsilon_n$ belong to $F$. We know that
\begin{align*}
\epsilon_n = \D^{(n)}(\phi(P(T))_{|T = 0}
\end{align*}
for all $n \ge 0$. If $n \ge 2$, then it follows from $(\ref{Equation-The-equation-for-phi(P(T))})$ and $(\ref{Equation-The-Hasse-derivatives-of-H(T)})$ that
\begin{align*}
\epsilon_n = \D^{(n)}\phi(P(T))_{|T = 0} &= \sum_{k = 0}^{\infty}a_k\D^{(n)}(H(T)^k)_{|T = 0} = a_0 + a_1\D^{(n)}(H(T))_{|T = 0} + a_n\gamma^n = a_0 + a_n\gamma^n.
\end{align*}
Since $P(T)$ is bounded, there is an absolute constant $c > 0$ such that $|a_k|_F < c$ for all $k \ge 0$. Since $|\gamma^n|_F = 1$ for all $n \ge 0$, we deduce that
\begin{align*}
|\epsilon_n|_F = |a_0 + a_n\gamma^n|_F \le \max(|a_0|_F, |a_n\gamma^n|_F) < c
\end{align*}
for all $n \ge 2$, which proves that $\phi(P(T))$ is bounded.

\end{proof}

\begin{remark}
\label{Remark-The-restriction-of-an-additive-isomorphism-to-the-algebra-of-bounded-power-series-is-an-automorphism}

By Lemma \ref{Lemma-The-restriction-of-an-additive-isomorphism-to-the-ring-of-bounded-power-series}, we know that for any additive isomorphism $\phi : F[[T]] \rightarrow F[[T]]$, the restriction of $\phi$ to $F_b[[T]]$ is is an automorphism of $F_b[[T]]$. By abuse of notation, we also denote by $\phi$ the restriction of $\phi$ to $F_b[[T]]$. To make clear which isomorphism we use, we sometimes write $\phi : F_b[[T]] \rightarrow F_b[[T]]$ to distinguish the restriction of $\phi$ to $F_b[[T]]$.

\end{remark}

\begin{definition}
\label{Definition-The-dual-of-a-flow}

Let $\cF : F \rightarrow \cU$, $\widehat{\cF} : F \rightarrow \cU$ be umbral maps. The flow map $\cD_{\cF}$ is called a \textit{dual} of $\cD_{\widehat{\cF}}$ if there exists an additive isomorphism $\phi : F[[T]] \rightarrow F[[T]]$ such that for any $x \in \cA_{\cF} \cap \cA_{\widehat{\cF}}$, the diagram
\begin{align}
\label{Equation-The-diagram-in-the-definition-of-a-dual-flow}
\begin{CD}
F_b[[T]]      @>\cD_{\cF}(x)>>                  F_b[[T]] \\
@VV\phi V                                       @VV\phi V \\
F_b[[T]]      @>\cD_{\widehat{\cF}}(x)>>        F_b[[T]]
\end{CD}
\end{align}
commutes.

When $\cD_{\cF}$ is a dual of $\cD_{\widehat{\cF}}$, we write $\cD_{\cF} \thicksim \cD_{\widehat{\cF}}$. For each $x \in \cA_{\cF} \cap \cA_{\widehat{\cF}}$, we say that the flow $\cD_{\cF}(x)$ is a dual of the flow $\cD_{\widehat{\cF}}(x)$.

\end{definition}

\begin{remark}

By Corollary \ref{Corollary-A-flow-is-an-action-on-the-algebra-of-bounded-power-series} and Lemma \ref{Lemma-The-restriction-of-an-additive-isomorphism-to-the-ring-of-bounded-power-series}, we see that the diagram defined by $(\ref{Equation-The-diagram-in-the-definition-of-a-dual-flow})$ makes sense.

\end{remark}

\begin{remark}
\label{Remark-A-dual-flow-is-produced-by-an-inner-automorphism}

Note that if $\cD_{\cF}$ is a dual of $\cD_{\widehat{\cF}}$, then there exists an additive isomorphism $\phi$ such that
\begin{align*}
\cD_{\widehat{\cF}}(x) = \phi \circ \cD_{\cF}(x) \circ \phi^{-1}
\end{align*}
for all $x \in \cA_{\cF} \cap \cA_{\widehat{\cF}}$.

\end{remark}

\begin{remark}
\label{Remark-Dual-is-almost-an-equivalence-relation}

It is obvious from Definition \ref{Definition-The-dual-of-a-flow} that the dual relation ``$\thicksim$" on the set of umbral maps is reflexive and symmetric. More explicitly, the binary relation ``$\thicksim$" satisfies the following.
\begin{itemize}

\item [(i)] $\cD_{\cF} \thicksim \cD_{\cF}$ for any umbral map $\cF$ (reflexivity); and

\item [(ii)] if $\cF, \widehat{\cF}$ are umbral maps such that $\cD_{\cF} \thicksim \cD_{\widehat{\cF}}$, then $\cD_{\widehat{\cF}} \thicksim \cD_{\cF}$ (symmetry).

\end{itemize}

\end{remark}

Restricting ``$\thicksim$" to a certain smaller subset of the set of all umbral maps, the binary relation ``$\thicksim$" becomes an equivalence relation.

\begin{proposition}
\label{Proposition-The-dual-relation-is-an-equivalence-relation}

Let $\cF : F \rightarrow \cU$ be an umbral map. Define
\begin{align*}
\cS_{\cF} := \{\text{$\cG : F \rightarrow \cU$ umbral maps} \; | \; \cA_{\cG} = \cA_{\cF}\}.
\end{align*}
Then the binary relation ``$\thicksim$" on the set $\cS_{\cF}$ is an equivalence relation.

\end{proposition}

\begin{proof}

The transitivity relation follows immediately from the definition of $\cS_{\cF}$, and hence it follows from Remark \ref{Remark-Dual-is-almost-an-equivalence-relation} that the binary relation ``$\thicksim$" is an equivalence relation.

\end{proof}

The following theorem is a generalization of \cite[Theorem 6]{Goss-JAlgebra-146-1992} that completely describes all umbral maps whose corresponding flow maps are dual to each other, and shows that the converse of \cite[Theorem 6]{Goss-JAlgebra-146-1992} is true.

\begin{theorem}
\label{Theorem-A-flow-under-the-Fourier-transform}

Let $\cF : F \rightarrow \cU$, $\widehat{\cF} : F \rightarrow \cU$ be umbral maps. Then the flow map $\cD_{\cF}$ is a dual of the flow map $\cD_{\widehat{\cF}}$ if and only if there exists an additive isomorphism $\phi$ such that
\begin{align*}
\widehat{\cF}_k(x) = \eval\left((H^{(-1)}(\cF(x)))^k\right)
\end{align*}
for all $k \ge 0$ and all $x \in \cA_{\cF} \cap \cA_{\widehat{\cF}}$, where $H(T) = \phi(T) \in F_b[[T]]$ is the generator of $\phi$ and $H^{(-1)}(T)$ is its composition inverse.

\end{theorem}

\begin{remark}
\label{Remark-Goss-theorem-is-the-if-part-of-the-theorem-about-flows-under-the-Fourier-transform}

The ``if" part of Theorem \ref{Theorem-A-flow-under-the-Fourier-transform} is due to Goss (see \cite[Theorem 6]{Goss-JAlgebra-146-1992}). Theorem \ref{Theorem-A-flow-under-the-Fourier-transform} says that the converse of Goss's theorem also holds. Hence Theorem \ref{Theorem-A-flow-under-the-Fourier-transform} completely describes umbral maps that are dual to each other, and signifies that the additive Fourier transform constructed in \cite{Goss-JAlgebra-146-1992} seems the most natural one.

\end{remark}

\begin{proof}[Proof of Theorem \ref{Theorem-A-flow-under-the-Fourier-transform}]

By Remark \ref{Remark-Goss-theorem-is-the-if-part-of-the-theorem-about-flows-under-the-Fourier-transform}, it suffices to prove the ``only if" part of Theorem \ref{Theorem-A-flow-under-the-Fourier-transform}.

Assume that $\cD_{\cF}$ is a dual of $\cD_{\widehat{\cF}}$. Then there exists an additive isomorphism $\phi$ such that the diagram $(\ref{Equation-The-diagram-in-the-definition-of-a-dual-flow})$ commutes, that is,
\begin{align}
\label{Equation-The-1st-equation-in-the-theorem-about-a-flow-under-the-Fourier-transform}
\phi\circ \cD_{\cF}(x) = \cD_{\widehat{\cF}}(x) \circ \phi
\end{align}
for all $x \in \cA_{\cF} \cap \cA_{\widehat{\cF}}$.

Let $H(T) = \phi(T) \in F_b[[T]]$ be the generator of $\phi$. We know that the composition inverse of $H(T)$, say $H^{(-1)}(T)$ is the generator for $\phi^{-1}$. Let $P(T) \in F_b[[T]]$ be an arbitrary bounded power series, and take an arbitrary element $x \in \cA_{\cF} \cap \cA_{\widehat{\cF}}$. By Theorem \ref{Theorem-A-generalization-of-the-Goss-theorem}, we see that
\begin{align*}
\phi\circ \cD_{\cF}(x)P(T) &= \phi(\cD_{\cF}(x)P(T))  \\
&= \phi\left(\eval(P(T + \cF(x)))\right) \\
&= \phi\left(\sum_{k = 0}^{\infty}\eval(\D^{(k)}P(\cF(x))) T^k\right)  \\
&= \sum_{k = 0}^{\infty}\eval(\D^{(k)}P(\cF(x))) \phi(T)^k,
\end{align*}
and hence
\begin{align}
\label{Equation-The-2nd-equation-in-the-theorem-about-a-flow-under-the-Fourier-transform}
\phi\circ \cD_{\cF}(x)P(T) = \sum_{k = 0}^{\infty}\eval(\D^{(k)}P(\cF(x))) H(T)^k.
\end{align}
On the other hand, we have
\begin{align*}
(\cD_{\widehat{\cF}}(x) \circ \phi)(P(T)) &= \cD_{\widehat{\cF}}(x)(\phi(P(T))) \\
&= \cD_{\widehat{\cF}}(x)((P\circ H)(T)) \\
&= \eval\left((P \circ H)(T + \widehat{\cF}(x))\right) \\
&= \eval\left(\sum_{k = 0}^{\infty}\D^{(k)}(P\circ H)(\widehat{\cF}(x)) T^k\right),
\end{align*}
and thus
\begin{align}
\label{Equation-The-3rd-equation-in-the-theorem-about-a-flow-under-the-Fourier-transform}
(\cD_{\widehat{\cF}}(x) \circ \phi)(P(T)) = \sum_{k = 0}^{\infty}\eval\left(\D^{(k)}(P\circ H)(\widehat{\cF}(x))\right) T^k.
\end{align}
By $(\ref{Equation-The-1st-equation-in-the-theorem-about-a-flow-under-the-Fourier-transform})$, $(\ref{Equation-The-2nd-equation-in-the-theorem-about-a-flow-under-the-Fourier-transform})$, $(\ref{Equation-The-3rd-equation-in-the-theorem-about-a-flow-under-the-Fourier-transform})$, we deduce that
\begin{align}
\label{Equation-The-4th-equation-in-the-theorem-about-a-flow-under-the-Fourier-transform}
\sum_{k = 0}^{\infty}\eval\left(\D^{(k)}P(\cF(x))\right) H(T)^k = \sum_{k = 0}^{\infty}\eval\left(\D^{(k)}(P\circ H)(\widehat{\cF}(x))\right) T^k
\end{align}
for all bounded power series $P(T) \in F_b[[T]]$. Since $H(T)$ is the generator of $\phi$, it is of the form
\begin{align*}
H(T) = \gamma T + \text{higher order terms}
\end{align*}
for some element $\gamma \in \cO_F^{\times}$. Thus $\eval(\D^{(0)}P(\cF(x))) = \eval(P(\cF(x)))$ is the coefficient of $T^0$ in the formal power series on the left-hand side of $(\ref{Equation-The-4th-equation-in-the-theorem-about-a-flow-under-the-Fourier-transform})$. Upon comparing the coefficients of $T^0$ on both sides of $(\ref{Equation-The-4th-equation-in-the-theorem-about-a-flow-under-the-Fourier-transform})$, we deduce that
\begin{align}
\label{Equation-The-5th-equation-in-the-theorem-about-a-flow-under-the-Fourier-transform}
\eval(P(\cF(x))) = \eval\left(\D^{(0)}(P\circ H)(\widehat{\cF}(x))\right) = \eval\left((P\circ H)(\widehat{\cF}(x))\right)
\end{align}
for all bounded power series $P(T) \in F_b[[T]]$.

For each $k \ge 0$, take $P(T) = P_k \circ H^{(-1)}(T)$ in $(\ref{Equation-The-5th-equation-in-the-theorem-about-a-flow-under-the-Fourier-transform})$, where $P_k(T) = T^k$. Hence we deduce from $(\ref{Equation-The-5th-equation-in-the-theorem-about-a-flow-under-the-Fourier-transform})$ that
\begin{align*}
\eval\left((H^{(-1)}(\cF(x)))^k\right) &= \eval\left(P_k \circ H^{(-1)}(\cF(x))\right) = \eval(P(\cF(x))) \\
&= \eval\left((P\circ H)(\widehat{\cF}(x))\right) = \eval\left(((P_k \circ H^{(-1)})\circ H)(\widehat{\cF}(x))\right) \\
&= \eval\left(P_k(\widehat{\cF}(x))\right) = \eval\left(\widehat{\cF}(x)^k\right) \\
&= \widehat{\cF}_k(x)
\end{align*}
for all $k \ge 0$. Since $x$ is arbitrary in $\cA_{\cF} \cap \cA_{\widehat{\cF}}$, Theorem \ref{Theorem-A-flow-under-the-Fourier-transform} follows immediately.

\end{proof}

\begin{definition}
\label{Definition-Umbrae-satisfy-the-binomial-theorem}

Let $\cF : F \rightarrow \cU$ be an umbral map, and let $S$ be a subset of $F$. The map $\cF$ is said to \textit{satisfy the binomial theorem with respect to $S$} if
\begin{align*}
\cF_n(x + y) = \sum_{k = 0}^{n} \binom{n}{k}\cF_k(x)\cF_{n - k}(y)
\end{align*}
for any $x, y \in S$.

\end{definition}

\begin{remark}
\label{Remark-S-is-closed-under-addition-if-the-umbra-F-satisfies-the-binomial-theorem-wrt-S}

Let $S$ be a subset of $\cA_F$. It then follows from Definition \ref{Definition-Umbrae-satisfy-the-binomial-theorem} that if $\cF$ satisfies the binomial theorem with respect to $S$, then $x + y$ belongs to $\cA_F$ for any $x, y \in S$.

\end{remark}

Keeping the same notation as in the above remark, we see that $x + y$ belongs to $\cA_F$ for any $x, y \in S$, and hence $\cD_{\cF}(x + y)$ is well-defined. The following result follows immediately from Definition \ref{Definition-The-flow-operator}.

\begin{lemma}
\label{Lemma-The-flow-operator-D-F-acts-in-the-same-way-as-the-exponential-operator}

Let $\cF : F \rightarrow \cU$ be an umbral map, and let $S$ be a subset of $\cA_F$. Assume that $\cF$ satisfies the binomial theorem with respect to $S$. Then
\begin{align*}
\cD_{\cF}(x + y) = \cD_{\cF}(x)\cD_{\cF}(y)
\end{align*}
for any $x, y \in S$.

\end{lemma}

\begin{theorem}

Let $\cF : F \rightarrow \cU$, $\widehat{\cF} : F \rightarrow \cU$ be umbral maps such that $\cD_{\cF} \thicksim \cD_{\widehat{\cF}}$. Let $S$ be a subset of $\cA_{\cF} \cap \cA_{\widehat{\cF}}$. Then $\cF$ satisfies the binomial theorem with respect to $S$ if and only if $\widehat{\cF}$ satisfies the binomial theorem with respect to $S$.

\end{theorem}

\begin{proof}

We only need to prove the ``only if" part of the theorem since the binary relation ``$\thicksim$" is symmetric. Assume that $\cF$ satisfies the binomial theorem with respect to $S$. It then follows from Lemma \ref{Lemma-The-flow-operator-D-F-acts-in-the-same-way-as-the-exponential-operator} that $\cD_{\cF}(x + y) = \cD_{\cF}(x)\cD_{\cF}(y)$ for any $x , y \in S$. Thus it follows from \cite[Corollary 3 and Corollary 4]{Goss-JAlgebra-146-1992} that $\widehat{\cF}$ satisfies the binomial theorem with respect to $S$.

\end{proof}

\section{Examples, and a question of Goss}
\label{Section-Examples-and-a-question-of-Goss}

In this section, we present some examples of flow maps, and discuss their properties in the language of the classically umbral calculus. These flow maps was already given by Goss \cite{Goss-JAlgebra-146-1992}. Another goal in this section is to formulate and partially answer a generalization of a question of Goss about flow maps. The main result in this section is a first step toward completely understanding the question of Goss. We begin by introducing a special type of umbral maps that plays a key role in this section.

\subsection{Geometric umbral maps.}
\label{Subsection-Geometric-flows}

In this subsection, we assume that $F$ is a complete field under a non-Archimedean norm $|\cdot|_F$ of any characteristic. We introduce the following notion that will play a key role in the rest of this section.

\begin{definition}
\label{Definition-Geometric-flows}

Let $\cF: F \rightarrow \cU$ be an umbral map, and for each $x \in F$, let $\{\cF_k(x)\}_{k \ge 0}$ be the sequence umbrally represented by $\cF(x)$. Let $x$ be an element in $F$. We say that $\cF$ is a \textit{geometric umbral map at $x$}, or equivalently $\cF$ is \textit{geometric at $x$} if $\cF_k(x) = \cF_1(x)^k$ for all $k \ge 0$.

When $\cF$ is geometric at any element $x$ in $\cA_{\cF}$, we simply say that $\cF$ is a \textit{geometric umbral map}, or that $\cF$ is \textit{geometric}.

\end{definition}

For a geometric umbral map $\cF$, the set $\cA_{\cF}$ of $\cF$-admissible elements has a very simple description.
\begin{proposition}
\label{Proposition-The-sets-of-admissible-elements-for-geometric-umbral-maps}

Let $\cF : F \rightarrow \cU$ be a geometric umbral map, and for each $x \in F$, let $\{\cF_k(x)\}_{k \ge 0}$ be the sequence umbrally represented by $\cF(x)$. Then
\begin{align*}
\cA_{\cF} = \{x \in F \; | \; |\cF_1(x)|_F < 1 \}.
\end{align*}

\end{proposition}

\begin{proof}

We see that $x \in \cA_{\cF}$ if and only if $\lim_{k \rightarrow \infty}|\cF_k(x)| = 0$. Since $\cF_k(x) = \cF_1(x)^k$ for all $k \ge 0$, the last condition is equivalent to saying that $|\cF_1(x)|_F < 1$, and hence Proposition \ref{Proposition-The-sets-of-admissible-elements-for-geometric-umbral-maps} follows.

\end{proof}

In the next two subsections, we will give some examples of geometric flow maps that are of great interest in this paper.

Let $\Gamma : F \rightarrow F$ be a function defined over $F$ and taking values in $F$. For each $x \in F$ with $|\Gamma(x)|_F < 1$, we denote by $e^{\Gamma(x)\D}$ the linear operator defined by
\begin{align*}
e^{\Gamma(x)\D} := \sum_{k \ge 0}\Gamma(x)^k \D^{(k)}.
\end{align*}

The operator $e^{\Gamma(x)\D}$ is well-defined if $x$ is an element in $F$ such that $|\Gamma(x)|_F < 1$. We can prove this fact by relating the above operator to a flow map. Indeed, let $\cF : F \rightarrow \cU$ be an umbral map such that $\eval\left(\cF(x)^k\right) = \Gamma(x)^k$ for any $x \in F$ and all $k \ge 0$. In other words, for each $x \in F$, the sequence $\{\cF_k(x)\}_{k \ge 0}$ is umbrally represented by $\cF(x)$, where $\cF_k(x) = \Gamma(x)^k$ for all $k \ge 0$. It is easy to see that $\cF$ is a geometric umbral map, and that any geometric umbral map can be constructed in the same way as $\cF$. When a geometric umbral map $\cF$ arises out of a function $\Gamma$ defined over $F$ and taking values in $F$ in the same way as presented above, we say that $\cF$ is the \textit{associated geometric umbral map of $\Gamma$}.

By Definition \ref{Definition-The-flow-operator}, we see that
\begin{align*}
\cD_{\cF}(x) = \sum_{k \ge 0}\cF_k(x)\D^{(k)} = \sum_{k \ge 0}\Gamma(x)^k\D^{(k)} = e^{\Gamma(x)\D}
\end{align*}
for all $x \in \cA_{\cF}$. By Proposition \ref{Proposition-The-sets-of-admissible-elements-for-geometric-umbral-maps} and since $\cF$ is geometric, we see that \begin{align*}
\cA_{\cF} = \{x \in F \; | \; |\Gamma(x)|_F < 1\}.
\end{align*}

The following result gives a simple description of the flow map of a geometric umbral map.

\begin{proposition}
\label{Proposition-Flow-maps-of-geometric-umbral-maps}

Let $\Gamma$ be a function defined over $F$ and taking values in $F$. Let $\cF : F \rightarrow \cU$ be the associated geometric umbral map of $\Gamma$. Then
\begin{align*}
e^{\Gamma(x)\D}P(T) = \cD_{\cF}(x)P(T) = P(T + \Gamma(x))
\end{align*}
for all $x \in F$ with $|\Gamma(x)|_F < 1$ and all $P(T) \in F_b[[T]]$.

\end{proposition}

\begin{proof}

Let $P(T) = \sum_{k \ge 0}a_k T^k \in F_b[[T]]$, and let $x \in F$ such that $|\Gamma(x)|_F < 1$. By the discussion preceding Proposition \ref{Proposition-Flow-maps-of-geometric-umbral-maps} and Theorem \ref{Theorem-A-generalization-of-the-Goss-theorem}, we see that
\begin{align*}
e^{\Gamma(x)\D}P(T) = \cD_{\cF}(x)P(T) = \sum_{k \ge 0}a_k \left(\sum_{m = 0}^k\binom{k}{m} \cF_{m}(x) T^{k - m}\right).
\end{align*}
Since $\cF$ is the associated geometric flow map of $\Gamma$, we deduce that
\begin{align*}
\sum_{k \ge 0}a_k \left(\sum_{m = 0}^k\binom{k}{m} \cF_{m}(x) T^{k - m}\right) &= \sum_{k \ge 0}a_k \left(\sum_{m = 0}^k\binom{k}{m} \Gamma(x)^m T^{k - m}\right) \\
&=  \sum_{k \ge 0}a_k (T + \Gamma(x))^k \\
&= P(T + \Gamma(x)),
\end{align*}
and hence Proposition \ref{Proposition-Flow-maps-of-geometric-umbral-maps} follows immediately.

\end{proof}

\begin{remark}
\label{Remark-Flow-maps-of-geometric-umbral-maps}

By the above proposition, we see that $\cD_{\cF}(x) = e^{\Gamma(x)\D}$ for all $x \in F$ with $|\Gamma(x)|_F < 1$. Hence by Corollary \ref{Corollary-A-flow-is-an-action-on-the-algebra-of-bounded-power-series}, we see that $e^{\Gamma(x)\D}$ is defined over $F_b[[T]]$ and taking values in $F_b[[T]]$ for all $x \in F$ with $|\Gamma(x)|_F < 1$.

\end{remark}

\subsection{The additive umbral map.}
\label{Subsection-The-additive-flow}

The following example was given by Goss in \cite[Example 1]{Goss-JAlgebra-146-1992}. In this example, we let the complete field $F$ as in Subsection \ref{Subsection-Geometric-flows}. Let $\cS : F \rightarrow \cU$ be an umbral map such that for each $x \in F$, the sequence $\{\cS_k(x)\}_{k \ge 0}$ umbrally represented by $\cS(x)$ is defined by
\begin{align*}
\cS_k(x) = x^k
\end{align*}
for all $k \ge 0$. The umbral map $\cS$ is called the \textit{additive umbral map}, and the corresponding flow map $\cD_{\cS}$ of $\cS$ is called the \textit{additive flow map}.

By the definition of $\cS$, we see that $\cS$ is a geometric umbral map, and hence by Proposition \ref{Proposition-The-sets-of-admissible-elements-for-geometric-umbral-maps}, we deduce that $\cA_{\cS} := \{x \in F \; | \; |x|_F < 1\}$. By the classical binomial theorem, we deduce immediately that $\cS$ satisfies the binomial theorem with respect to $\cA_F$.

For the rest of this section, the symbol $\cS$ always denotes the additive flow map.

\subsection{The naive umbral map.}
\label{Subsection-The-naive-flow}

In this subsection, we let $A := \bF_q[t]$, $k := \bF_q(t)$, and let $F := \bF_q((1/t))$ be the completion of $k$, where $q$ is a power of a prime $p$. In what follows, we recall the naive flow map that was first introduced by Goss \cite{Goss-JAlgebra-146-1992}. Let $C$ be the Carlitz module, and let $e_C(x)$ be the exponential of $C$ (see \cite{Goss}). Recall from \cite{Goss} that
\begin{align*}
e_C(x) = \sum_{k \ge 0}\dfrac{x^{q^k}}{D_k},
\end{align*}
where $D_k$ is the product of all monic polynomials in $A$ of degree $k$ for each $k \ge 0$. Let $\cN : F \rightarrow \cU$ be the associated geometric umbral map of $e_C(x)$. The umbral map $\cN$ is called the \textit{naive umbral map}, and the corresponding flow map $\cD_{\cN}$ of $\cN$ is called the \textit{naive flow map}. Note that $|e_C(x)|_F < 1$ when $|x|_F < 1$. The following result follows immediately from Proposition \ref{Proposition-Flow-maps-of-geometric-umbral-maps}.

\begin{proposition}
\label{Proposition-The-naive-flow-map}

For all $x \in F$ with $|x|_F < 1$, the naive flow map satisfies the equation
\begin{align*}
\cD_{\cN}(x)P(T) = e^{e_C(x)\D}P(T) = P(T + e_C(x))
\end{align*}
for all bounded power series $P(T) \in F_b[[T]]$.

\end{proposition}

For the rest of this section, we always denote by $\cN$ the naive umbral map.

\subsection{The twisted flow map.}
\label{Subsection-The-twisted-flow}

We let the polynomial ring $A$, the function field $k$ and the complete field $F$ as in Subsection \ref{Subsection-The-naive-flow}. For each $k \ge 0$, we define
\begin{align*}
e_k(x) = \prod_{\substack{\epsilon \in A, \; \mathrm{deg}(\epsilon) < k}}(x + \epsilon).
\end{align*}
For each integer $n \ge 0$, write $n$ $q$-adically as $n = \sum_{k = 0}^{h}\epsilon_k q^k$ with $0 \le \epsilon_k < q$, and define
\begin{align*}
\cT_n(x) := \prod_{k = 0}^{h}\left(\dfrac{e_k(x)}{D_k}\right)^{\epsilon_k},
\end{align*}
where the $D_k$ are the same as in Subsection \ref{Subsection-The-naive-flow}. Let $\cT : F \rightarrow \cU$ be the umbral map such that for each $x \in F$, the sequence $\{\cT_n(x)\}_{n \ge 0}$ is umbrally represented by $\cT(x)$. We call $\cT$ the \textit{twisted umbral map}, and the corresponding flow map $\cD_{\cT}$ the \textit{twisted flow map}. The twisted flow map was first introduced by Goss (see \cite[Example 1]{Goss-JAlgebra-146-1992}). As was shown in \cite{Goss-KTheory-2-1989}, it is well-known that $\cA_{\cT} = F$.

As was pointed out by Goss \cite{Goss-JAlgebra-146-1992}, Carlitz proved that the twisted umbral map $\cT$ satisfies the binomial theorem with respect to $F$.

Note further that $\cT$ is not geometric at any $x \in A^{\times}$. Indeed assume the contrary, that is, $\cT$ is geometric at some element $x \in A^{\times}$. Then it follows that $\cT_n(x) = \cT_1(x)^n = x^n$ for all $n \ge 0$. But $T_n(x) = 0$ for a sufficiently large integer $n$, which is a contradiction. Therefore $\cT$ is not geometric at any nonzero element $x \in A^{\times}$.

\subsection{A question of Goss}
\label{Subsection-A-question-of-Goss}

In this subsection, we give a partial answer to an old question of Goss \cite{Goss-JAlgebra-146-1992}. Throughout this subsection, $F$ is a complete field under a non-Archimedean norm $|\cdot|_F$ of characteristic $p > 0$. We begin by recalling Goss's question.

\begin{question}
\label{Question-A-question-of-Goss}
$(\text{Goss})$

Let $\cF : F \rightarrow \cU$, $\widehat{\cF} : F \rightarrow \cU$ be umbral maps such that the flow map $\cD_{\cF}$ is dual to the flow map $\cD_{\widehat{\cF}}$. Let $\phi : F[[T]] \rightarrow F_b[[T]]$ be an additive isomorphism such that the diagram $(\ref{Definition-The-dual-of-a-flow})$ commutes, that is, $\cD_{\widehat{\cF}}(x) = \phi \circ \cD_{\cF} \circ \phi^{-1}$ for all $x \in \cA_{\cF} \cap \cA_{\widehat{\cF}}$. By Theorem \ref{Theorem-A-flow-under-the-Fourier-transform}, we know that
\begin{align}
\label{Equation-The-equation-between-two-flow-maps-in-the-question-of-Goss}
\widehat{\cF}_k(x) = \eval\left(H^{(-1)}(\cF(x))^k\right)
\end{align}
for all $k \ge 0$ and all $x \in \cA_{\cF} \cap \cA_{\widehat{\cF}}$, where $H^{(-1)}(T)$ is the composition inverse of the additive bounded power series $H(T) := \phi(T)$. Let $\cA_{\widehat{\cF}_1}$ denote the set of all element $x \in F$ such that $|\widehat{\cF}_1(x)| < 1$. Goss \cite{Goss-JAlgebra-146-1992} asked what the exact relationship between the flow $e^{\widehat{\cF}_1(x)\D}$ and the flow $\cD_{\widehat{\cF}}(x)$ for each $x \in \cA_{\cF} \cap \cA_{\widehat{\cF}} \cap \cA_{\widehat{\cF}_1}$ is. (Note that since $|\widehat{\cF}_1(x)| < 1$ for each $x \in \cA_{\cF} \cap \cA_{\widehat{\cF}} \cap \cA_{\widehat{\cF}_1}$, the operator $e^{\widehat{\cF}_1(x)\D}$ is well-defined.)

\end{question}

The rest of this section is to prove a necessary and sufficient condition for which the two flows in Question \ref{Question-A-question-of-Goss} are equal to each other. It turns out that the flow $e^{\widehat{\cF}_1(x)\D}$ is equal to the flow $\cD_{\widehat{\cF}}(x)$ for each $x \in \cA_{\cF} \cap \cA_{\widehat{\cF}} \cap \cA_{\widehat{\cF}_1}$ if and only if $\cF$ is a geometric umbral map at any element $x \in \cA_{\cF} \cap \cA_{\widehat{\cF}} \cap \cA_{\widehat{\cF}_1}$. This result is a first step toward fully understanding the question of Goss. It would be very interesting to know what the exact relationship between the two flows is when $\cF$ is not a geometric umbral map. We now prove the main result in this section.

\begin{theorem}
\label{Theorem-A-theorem-about-Goss-question}

We maintain the same notation and assumptions as in Question \ref{Question-A-question-of-Goss}. Let $x$ be an element in $\cA_{\cF} \cap \cA_{\widehat{\cF}} \cap \cA_{\widehat{\cF}_1}$. Then the flow $e^{\widehat{\cF}_1(x)\D}$ is equal to the flow $\cD_{\widehat{\cF}}(x)$ if and only if the umbral map $\cF$ is geometric at $x$.

\end{theorem}

We first prove some lemmas that we will need in the proof of Theorem \ref{Theorem-A-theorem-about-Goss-question}. Using the same arguments as in the proof of Proposition \ref{Proposition-A-condition-for-a-bounded-power-series-to-be-admissible}, the following lemma is immediate.

\begin{lemma}
\label{Lemma-The-evaluation-of-a-bounded-power-series-at-an-umbra-is-admissible}

Let $P(T)$ be a bounded power series in $F_b[[T]]$. Let $\alpha$ be an umbra such that the sequence $\{u_k\}_{k \ge 0}$ umbrally represented by $\alpha$ satisfies $\lim_{k \rightarrow \infty}u_k = 0$. Then $P(\alpha) \in F_b[[T]][[\alpha]]$ is an admissible power series.

\end{lemma}

\begin{lemma}
\label{Lemma-The-evaluation-of-a-bounded-power-series-at-a-geometric-flow}

Let $\cF : F \rightarrow \cU$ be an umbra map, and let $P(T)$ be a bounded power series in $F_b[[T]]$. Let $x$ be an element in $\cA_{\cF}$ such that $\cF$ is geometric at $x$. Then $\eval(P(\cF(x)) = P(\cF_1(x)).$

\end{lemma}

\begin{proof}

Write
\begin{align*}
P(T) = \sum_{k \ge 0}a_kT^k,
\end{align*}
where the $a_k$ are elements in $F$. Since $P(T)$ is bounded and $x$ belongs to $\cA_{\cF}$, it follows from Lemma \ref{Lemma-The-evaluation-of-a-bounded-power-series-at-an-umbra-is-admissible} that $P(\cF(x))$ is an admissible power series, and hence $\eval(P(\cF(x))$ is well-defined as an element in $F$. We see that
\begin{align*}
P(\cF(x)) = \sum_{k \ge 0}a_k\cF(x)^k,
\end{align*}
and hence
\begin{align*}
\eval(P(\cF(x))) = \sum_{k \ge 0}a_k\eval\left(\cF(x)^k\right) = \sum_{k \ge 0}a_k\cF_k(x).
\end{align*}
Since $\cF$ is geometric at $x$, we know that $\cF_k(x) = \cF_1(x)^k$ for all $k \ge 0$. Thus it follows from the above equation that
\begin{align*}
\eval(P(\cF(x))) = \sum_{k \ge 0}a_k\cF_1(x)^k = P(\cF_1(x)),
\end{align*}
which proves our contention.

\end{proof}

\begin{lemma}
\label{Lemma-The-evaluation-of-powers-of-a-bounded-power-series-at-a-geometric-flow}

Let $\cF : F \rightarrow \cU$ be an umbra map, and let $P(T)$ be a bounded power series in $F_b[[T]]$. Let $x$ be an element in $\cA_{\cF}$ such that $\cF$ is geometric at $x$. Then
\begin{align*}
\eval(P(\cF(x))^k) = \eval(P(\cF(x)))^k
\end{align*}
for all $k \ge 0$.

\end{lemma}

\begin{remark}
\label{Remark-Powers-of-bounded-power-series-is-bounded}

It is easy to see that if $P(T)$ is a bounded power series, then $P(T)^n$ is bounded for all $n \ge 0$. Since $\lim_{k \rightarrow \infty}\cF_k(x) = 0$, we deduce from Lemma \ref{Lemma-The-evaluation-of-a-bounded-power-series-at-an-umbra-is-admissible} that $P(\cF(x))^k$ is an admissible power series for all $x \in \cA_{\cF}$.

\end{remark}

\begin{proof}

Take any integer $k \ge 0$, and define $P_k(T) = T^k \in F_b[[T]]$. By Remark \ref{Remark-Powers-of-bounded-power-series-is-bounded}, we know that $(P_k\circ P)(T) = P(T)^k$ is a bounded power series. Applying Lemma \ref{Lemma-The-evaluation-of-a-bounded-power-series-at-a-geometric-flow}, we deduce that
\begin{align*}
\eval\left(P(\cF(x))^k\right) = \eval\left((P_k\circ P)(\cF(x))\right) = (P_k\circ P)(\cF_1(x)) = P(\cF_1(x))^k,
\end{align*}
and thus the lemma follows.

\end{proof}

Now we are ready to prove Theorem \ref{Theorem-A-theorem-about-Goss-question}

\begin{proof}[Proof of Theorem \ref{Theorem-A-theorem-about-Goss-question}]

We maintain the same notation as in Question \ref{Question-A-question-of-Goss}.

If $e^{\widehat{\cF}_1(x)\D}$ is equal to $\cD_{\widehat{\cF}}(x)$, we see that
\begin{align*}
\sum_{k \ge 0}\eval\left(H^{(-1)}(\cF(x))\right)^k \D^{(k)} = e^{\widehat{\cF}_1(x)\D} = \cD_{\widehat{\cF}}(x) = \sum_{k \ge 0}\eval\left(H^{(-1)}(\cF(x))^k\right)\D^{(k)}.
\end{align*}
Thus we deduce that
\begin{align}
\label{Equation-The-first-equation-in-the-theorem-about-Goss-question}
\eval\left(H^{(-1)}(\cF(x))\right)^k = \eval\left(H^{(-1)}(\cF(x))^k\right)
\end{align}
for all $k \ge 0$. Set
\begin{align}
\label{Equation-The-symbol-Q1-in-the-theorem-about-Goss-question}
\cQ_1(x) := H^{(-1)}(\cF(x)).
\end{align}
By Lemma \ref{Lemma-The-evaluation-of-a-bounded-power-series-at-an-umbra-is-admissible} and since $H^{(-1)}(T)$ is bounded, the power series $\cQ_1(x) \in F_b[[T]][[\cF(x)]]$ is admissible, and thus $\eval(\cQ_1(x))$ is well-defined as an element in $F$. Since $x$ belongs to $\cA_{\widehat{\cF}_1}$, we see that
\begin{align*}
\left|\eval(\cQ_1(x))\right|_F = \left|\eval(H^{(-1)}(\cF(x)))\right|_F = \left|\widehat{\cF}_1(x)\right|_F < 1.
\end{align*}

Equation $(\ref{Equation-The-first-equation-in-the-theorem-about-Goss-question})$ is equivalent to the equation
\begin{align}
\label{Equation-The-2nd-equation-in-the-theorem-about-Goss-question}
\eval\left(\cQ_1(x)^k\right) = \eval\left(\cQ_1(x)\right)^k = \widehat{\cF}_1(x)^k
\end{align}
for all $k \ge 0$.

It follows from $(\ref{Equation-The-symbol-Q1-in-the-theorem-about-Goss-question})$ that
\begin{align*}
H(\cQ_1(x))^n = H(H^{(-1)}(\cF(x)))^n = \cF(x)^n
\end{align*}
for all $n \ge 0$, which implies that $H(\cQ_1(x))^n$ is an admissible power series for all $n \ge 0$. Now take any integer $n \ge 0$, and write
\begin{align*}
H(T)^n = \sum_{k \ge 0}h_k T^k,
\end{align*}
where the $h_k$ are elements in $F$. Hence we deduce that
\begin{align*}
\cF(x)^n = H(\cQ_1(x))^n = \sum_{k \ge 0}h_k \cQ_1(x)^k,
\end{align*}
and it thus follows from $(\ref{Equation-The-2nd-equation-in-the-theorem-about-Goss-question})$ that
\begin{align*}
\cF_n(x) = \eval\left(\cF(x)^n\right) = \eval\left(\sum_{k \ge 0}h_k \cQ_1(x)^k\right) = \sum_{k \ge 0}h_k \eval\left(\cQ_1(x)^k\right) = \sum_{k \ge 0}h_k \widehat{\cF}_1(x)^k = H(\widehat{\cF}_1(x))^n.
\end{align*}
When $n$ equals $1$, we see that $\cF_1(x) = H(\widehat{\cF}_1(x))$, and thus it follows from the above equation that
\begin{align*}
\cF_n(x) = H(\widehat{\cF}_1(x))^n = \cF_1(x)^n,
\end{align*}
which proves that $\cF$ is geometric at $x$.

Now we prove the backward implication in Theorem \ref{Theorem-A-theorem-about-Goss-question}. Indeed, if $\cF$ is geometric at $x$, then applying Lemma \ref{Lemma-The-evaluation-of-a-bounded-power-series-at-a-geometric-flow} with $H^{(-1)}(T)$ in the role of $P(T)$, we deduce that
\begin{align*}
\eval\left(H^{(-1)}(\cF(x))\right)^k = \eval\left(H^{(-1)}(\cF(x))^k\right)
\end{align*}
for all $k \ge 0$. Thus $e^{\widehat{\cF}_1(x)\D}(x)$ is equal to $\cD_{\widehat{\cF}}(x)$. Therefore our contention follows.

\end{proof}

\begin{example}
\label{Example-The-naive-flow-is-the-image-of-the-additive-flow-under-the-Fourier-transform}

Throughout this example, we let $F$ be the complete field as in Subsection \ref{Subsection-The-naive-flow}. Let $\cS$ be the additive umbral map in Subsection \ref{Subsection-The-additive-flow}, and let $\cN$ be the naive umbral map in Subsection \ref{Subsection-The-naive-flow}. Since the power series $H(T) := e_C(T)$ satisfies all conditions in Definition \ref{Definition-Additive-linear-isomorphisms}, we see that the isomorphism $\phi : F[[T]] \rightarrow F[[T]]$ defined by $\phi(T) = H(T)$ is an additive isomorphism. Note that
\begin{align*}
\cA_{\cS_1} = \{x \in F \; | \; |\cS_1(x)|_F = |x|_F < 1\} = \cA_{\cS}.
\end{align*}

Since $\cS$ is geometric at any $x \in \cA_{\cS}$, we see that
\begin{align*}
\cN_n(x) = e_C(x)^n = H(\cS_1(x))^n = \eval\left(H(\cS(x))^n\right)
\end{align*}
for all $n \ge 0$ and any $x \in \cA_{\cS} \cap \cA_{\cN}$. By Theorem \ref{Theorem-A-flow-under-the-Fourier-transform}, we deduce that $\cD_{\cN}$ is dual to $\cD_{\cS}$. By Theorem \ref{Theorem-A-theorem-about-Goss-question}, we know that
\begin{align*}
\cD_{\cN}(x) = e^{\cN_1(x)\D} = e^{e_C(x)\D}
\end{align*}
for all $x \in \cA_{\cS} \cap \cA_{\cN}$. This is another way of constructing the naive flow map via the duality with the additive flow map.

\end{example}

\begin{example}
\label{Example-The-twitsed-flow-map-is-not-equal-to-the-naive-flow-map-under-any-Fourier-transform}

In this example, we let the complete field $F$ as in Subsection \ref{Subsection-The-naive-flow} and Subsection \ref{Subsection-The-twisted-flow}. We know from Subsection \ref{Subsection-The-twisted-flow} that the twisted umbral map $\cT$ is not geometric at any element $x \in A^{\times}$. Hence if we take any umbral map $\widehat{\cT} : F \rightarrow \cU$ such that the flow map $\cD_{\widehat{\cT}}(x)$ is dual to $\cD_{\cT}$, then we know from Theorem \ref{Theorem-A-theorem-about-Goss-question} that the flow $e^{\widehat{\cT}(x)\D}$ is not equal to the flow $\cD_{\widehat{\cT}}(x)$ for any $x \in \cA_{\cT} \cap \cA_{\widehat{\cT}} \cap \cA_{\widehat{\cT}_1}$.

For example, let
\begin{align*}
H(T) := \sum_{k \ge 0}(e_C(1) T)^{q^k} \in F_b[[T]].
\end{align*}
It is clear that $H(T)$ satisfies all the conditions in Definition \ref{Definition-Additive-linear-isomorphisms}. Let $\phi : F[[T]] \rightarrow F[[T]]$ be the additive isomorphism such that $\phi(T) = H(T)$. Let $\widehat{\cT} : F \rightarrow \cU$ be the umbral map such that the sequence $\{\widehat{\cT}_k(x)\}_{k \ge 0}$ umbrally represented by $\widehat{\cT}(x)$ is defined by
\begin{align*}
\widehat{\cT}_k(x) = \eval\left(H(\cT(x))^k\right)
\end{align*}
for all $k \ge 0$ and all $x \in \cA_{\cT} = F$. By Theorem \ref{Theorem-A-flow-under-the-Fourier-transform}, we know that the flow map $\cD_{\widehat{\cT}}$ is dual to the flow map $\cD_{\widehat{\cT}}$.

As was shown by Goss \cite{Goss-JAlgebra-146-1992}, we know that
\begin{align*}
\widehat{\cT}_1(x) = \eval(H(\cT(x)) = e_C(x).
\end{align*}
By Theorem \ref{Theorem-A-theorem-about-Goss-question} and since $\cT$ is not geometric at any $x \in A^{\times}$, we know that the naive flow $e^{\widehat{\cT}_1(x)\D} = e^{e_C(x)\D}$ is not equal to the flow $\cD_{\widehat{\cT}}(x)$ for any $x \in \cA_{\cT} \cap \cA_{\widehat{\cT}} \cap \cA_{\widehat{\cT}_1}$, where
\begin{align*}
\cA_{\widehat{\cT}_1} = \{x \in F \; | \; |\widehat{\cT}_1(x)|_F = |e_C(x)|_F < 1\}.
\end{align*}

\end{example}

\section{Epilogue}
\label{S-Epilogue}

We end this paper by making some comments on the umbral calculus that is introduced in this paper. We also compare our umbral calculus with other classical umbral calculi in Roman and Rota \cite{Roman-Rota}, Rota and Taylor \cite{Rota-Taylor-SIAM-25-1994}, and Ueno \cite{Ueno-1988}.

There are many ways of representing a sequence of numbers. The classical umbral calculus was born out of the realization that one can describe a sequence as a definite integral. For example, let $\{\alpha_k\}_{k \ge 0}$ be a sequence of real numbers. Choosing an appropriate function $F$, one can represent $\{\alpha_k\}_{k \ge 0}$ as
\begin{align*}
\alpha_k = \int_0^1\alpha^kF(\alpha)d\alpha.
\end{align*}
In terms of linear operators, the sequence $\{\alpha_k\}_{k \ge 0}$ is obtained by applying the linear functional $\cL$ to the sequence of polynomials $\{\alpha^k\}_{k \ge 0}$, where $\cL : \bR[\alpha] \rightarrow \bR$ is defined by
\begin{align*}
\cL(\sum_{i = 0}^nu_i\alpha^i) =\sum_{i = 0}^n u_i\int_0^1\alpha^iF(\alpha)d\alpha.
\end{align*}
In this way, studying the linear functional $\cL$ may offer an insight into the properties of $\{\alpha_k\}_{k \ge 0}$. 

For a collection of sequences $\left(\{\alpha_k^{(i)}\}_{k \ge 0}\right)_{i \in I}$, one can also use the above method to associate the collection to a linear functional. Indeed, for each $i \in I$, let $\cL_i : \bR[\gamma_i] \rightarrow \bR$ be the linear functional associated to the sequence $\{\alpha_k^{(i)}\}_{k \ge 0}$. Let $\cU$ be the collection of variables $\{\gamma_i\}_{i \in I}$. We define the linear functional $\cL : \bR[\cU] \rightarrow \bR$ such that
\begin{align*}
\cL(\gamma_{i_1}^{n_1}\cdots \gamma_{i_l}^{n_l}) = \cL_{i_1}(\gamma_{i_1}^{n_1})\cdots \cL_{i_l}(\gamma_{i_l}^{n_l})
\end{align*}
for any $\gamma_{i_1}, \ldots, \gamma_{i_l} \in \cU$. In this way, each sequence $\{\alpha_k^{(i)}\}_{k \ge 0}$ is obtained by applying the linear functional $\cL$ to the sequence of polynomials $\{\gamma_i^k\}_{k \ge 0}$. Thus obtaining information about the linear functional $\cL$ can give an insight into the properties of each sequence $\{\alpha_k^{(i)}\}_{k \ge 0}$ in the collection. One of course can replace $\bR$ by any integral commutative domain, and hence obtaining the umbral calculus studied in Roman and Rota \cite{Roman-Rota}, and Rota and Taylor \cite{Rota-Taylor-1993}.

The classical umbral calculus in \cite{Roman-Rota} and \cite{Rota-Taylor-1993} is very useful for studying special sequences of numbers such as the Bernoulli numbers (see Rota and Taylor \cite[Section {\bf4}]{Rota-Taylor-1993}).

In more algebraic terms, we can summarize the umbral calculus in \cite{Roman-Rota} and \cite{Rota-Taylor-1993} as follows. For any linear functionals $\cL_1, \cL_2$ on polynomials, one can define the multiplication ``$\star$'' between $\cL_1, \cL_2$ by
\begin{align*}
(\cL_1 \star\cL_2)(\alpha^n) = \sum_{i = 0}^n\binom{n}{i}\cL_1(\alpha^i)\cL_2(\alpha^{n - i}).
\end{align*}
The linear functionals on polynomials equipped with the usual addition and the multiplication ``$\star$'' form an algebra which Roman and Rota \cite{Roman-Rota} called the \textit{umbral algebra}. It is well-known \cite{Roman-Rota} that this algebra is topologically isomorphic to the algebra of formal power series. Hence from a more algebraic point of view, the classical umbral calculus is equivalent to studying the pair of topological $\fR$-linear spaces $(\fR[\alpha], \fR[[\tau]])$ in which the duality between these two spaces plays a central role, where $\fR$ is an integral domain of characteristic zero.

Ueno \cite{Ueno-1988} called the classical umbral calculus in \cite{Roman-Rota} and \cite{Rota-Taylor-1993} the \textit{polynomial umbral calculus}, partly because of its central role in studying classical polynomial sequences such as Hermite and Laguerre polynomials (see Roman and Rota \cite[Section {\bf13}]{Roman-Rota}). Other types of umbral calculi appear as the result of studying other pairs of topological linear spaces with duality. Ueno \cite{Ueno-1988} developed the \textit{general power umbral calculus} by studying either the pair of topological $F$-linear spaces $(\alpha^uF((\alpha)), \tau^u F((\tau^{-1})))$ or $(\lambda^uF((\lambda^{-1})), \gamma^u F((\gamma)))$ with duality. Here $u$ is an element in a field $F$ of characteristic zero. The general power umbral calculus plays a crucial role in studying classical special functions such as hypergeometric functions, and Bessel functions (see Ueno \cite[Part {\bf I}, Section {\bf7}]{Ueno-1988}).

Note that each umbral calculus is designed for a \textit{special} object of study. For example, the polynomial umbral calculus in \cite{Roman-Rota} and \cite{Rota-Taylor-1993} is suitable for studying special classical polynomial sequences, whereas the general power umbral calculus developed by Ueno \cite{Ueno-1988} is used to study classical special functions. In this paper, we are mainly devoted to studying flows in finite characteristic with applications to additive harmonic analysis in mind. Hence our umbral calculus is especially designed for studying flows in finite characteristic.

For a field $F$ complete under a non-Archimedean norm, a flow is an operator acting on the algebra  $F_b[[T]]$ of bounded power series. This is one of the reasons why we take $F_b[[T]]$ to be the range of the evaluation map $\eval$ in our umbral calculus. We also incorporate the non-Archimedean topology of $F$ into our umbral calculus, and slightly modify the notion of admissible power series that seems more natural and suitable for applications in finite characteristic. This allows us to extend the domain of the evaluation map $\eval$ to a sufficiently large domain in which our umbral calculus can be applied to studying flows in finite characteristic. 

We emphasize that there should be other umbral calculi in finite characteristic waiting to be discovered. For example, there is a function field analogue of Bernoulli numbers that is called Bernoulli--Carlitz numbers. Using Rota and Taylor \cite{Rota-Taylor-1993}, Carlitz \cite{Carlitz-Duke-Math-Journal-1935-no.2} \cite{Carlitz-Duke-Math-Journal-1937} \cite{Carlitz-Duke-Math-Journal-1940-von-Staudt}, and Goss \cite{Goss-Duke-1978}, can one find an umbral calculus to study Bernoulli--Carlitz numbers?

\section*{Acknowledgements}

I am very grateful to David Goss for his insightful comments and suggesting me to look at his question in Goss \cite{Goss-JAlgebra-146-1992}, which motivated me to write up Section \ref{Section-Examples-and-a-question-of-Goss} of this paper as a first step toward completely understanding his question. I also thank David Goss for many wonderful discussions and his encouragement. I would like to thank the referee for his useful comments, and many great questions that he posed to me. I really hope to return to these questions in a near future.

\end{document}